\title{A Riemannian Corollary of Helly's Theorem}
\author{
        Alexander Rusciano\\
                Department of Mathematics\\
        University of California, Berkeley\\
        rusciano@berkeley.edu
        }
\date{\today}

\documentclass[11pt]{article}
\usepackage{geometry}

\usepackage{amsfonts}
\usepackage{amsthm}
\usepackage{bbm}
\usepackage{amsmath}
\usepackage{amscd}
\usepackage[latin2]{inputenc}
\usepackage{t1enc}
\usepackage[mathscr]{eucal}
\usepackage{indentfirst}
 
\theoremstyle{plain}
\newtheorem{theorem}{Theorem}
\newtheorem{lemma}[theorem]{Lemma}

\newtheorem{proposition}[theorem]{Proposition}

\theoremstyle{plain}
\newtheorem{definition}[theorem]{Definition}

\DeclareMathOperator*{\argmin}{arg\,min}

\DeclareMathOperator\supp{supp}

\begin{document}

\maketitle

\begin{abstract}
We introduce a notion of halfspace for Hadamard manifolds that is natural in the context of convex optimization.  For this notion of halfspace, we generalize a classic result of Gr\"{u}nbaum, which itself is a corollary of Helly's theorem.  Namely, given a probability distribution on the manifold, there is a point for which all halfspaces based at this point have at least $\frac{1}{n+1}$ of the mass.  As an application, the subgradient oracle complexity of convex optimization is polynomial in the parameters defining the problem.
\end{abstract}

\section{Overview}
\subsection{Introduction}
The extrema of functions are of fundamental importance in mathematics and its applications.  Much of numerical optimization studies this topic.  Most of the theory focuses on convex functions, as it has proven hard to find other classes that are both useful and tractable.  The motivation for this paper comes from the desire to expand the boundaries of this class of tractable functions.

\vspace*{.3cm}
Rigorous study of convergence rates was initiated in \cite{ZS} for first order methods for convex functions on Hadamard manifolds. That is, gradient descent methods for simply connected manifolds of non-positive sectional curvature.  Such manifolds are diffeomorphic to $\mathbb{R}^n$ and exhibit natural convex functions. In a sense, they give new classes of functions for which optimization is tractable.

\vspace*{.3cm}
Still, as far as the author is aware, all known algorithms for general convex optimization on Riemannian manifolds have iteration complexity depending polynomially on $\epsilon^{-1}$.  To achieve better convergence rates, further conditions are added such as strong convexity, dominated gradients, or recently robust second-order \cite{ZS}, \cite{ZRS}, \cite{AGLOW}.  One major unresolved question for Hadamard manifolds like $SL_n / SO_n$ is, does convexity enable algorithms whose time complexity depends polynomially on $\log(\epsilon^{-1})$?

\vspace*{.3cm}
For Euclidean optimization, cutting plane methods are the standard, general approach to get $\log(\epsilon^{-1})$ complexity.  It is well known that the minimum of a convex function lies in the halfspace opposite the subgradient direction.  Cutting plane methods use this fact to reduce the feasible set.  One feature of $\mathbb{R}^n$ that enables this approach to succeed is the existence of what are commonly termed centerpoints.  A precise definition of centerpoint is given in Definition \ref{def:centerpoint}.  Roughly speaking, if $c$ is to be a centerpoint for a set $S$, then no halfspace based at $c$ should contain too large (or small) a fraction of the volume of $S$.  Ellipsoid methods explicitly maintain a radially symmetric set, so the center of the current ellipse provides a perfect centerpoint.  Thus a subgradient at the center of an ellipse allows one to eliminate half of the ellipse.  For more general subsets $S \subset \mathbb{R}^n$ than ellipses, Gr\"{u}nbaum's result in \cite{G} shows the existence of a centerpoint $c$ for which any halfspace based at $c$ contains at most a $\frac{n}{n+1}$-fraction of the mass of $S$.  

\vspace*{.3cm}
As generalizing this result of Gr\"{u}nbaum is the main goal of this work and we fundamentally build from it in the proof, we summarize his result as Theorem \ref{thm:classic}.  Note the statement is for a general probability measure, a fact we will make use of by applying it to the Riemannian volume measure.

\begin{theorem}\label{thm:classic}
A $\frac{1}{n+1}$-centerpoint $c$ exists for any probability measure on $\mathbb{R}^n$, endowed with the usual Borel $\sigma$-algebra.  Here $c$ being a $\frac{1}{n+1}$-centerpoint means any halfspace based at $c$ contains at least a $\frac{1}{n+1}$-fraction of the mass of the probability distribution.
\end{theorem}

\vspace*{.3cm}
We replicate this result in the more general setting of Hadamard manifolds in the hope that others find the result encouraging, useful, or intrinsically interesting.  Though we present the results here to provide motivation, the definitions in the Section \ref{sec:defs} make the statements precise.  The main result is from Section \ref{sec:centerpoint_existence},

\begin{theorem}\label{thm:sharpness}
Suppose $\mu$ is a probability distribution on a Hadamard manifold $M$ of dimension $n$, and $\mu$ is absolutely continuous with respect to the Riemannian volume measure $\text{vol}_g$.  Then there exists a $\frac{1}{n+1}$-centerpoint $c$ for the measure $\mu$.  If we assume the support of $\mu$ is contained in a (geodesic) convex set $S$, then $c \in S$.  Moreover, even for the uniform measure on convex and compact $S$, this value $\frac{1}{n+1}$ cannot in general be improved.
\end{theorem}

We noted the sharpness of $\frac{1}{n+1}$ in the above theorem because this contrasts with the guaranteed existence of $\frac{1}{e}$-centerpoints for the uniform measure on convex subsets of Euclidean space \cite{G}. 

\vspace*{.3cm}
The theorem leads to a bound on the number of subgradient oracle calls needed to optimize a function.

\begin{theorem}\label{thm:oracle_cuts}
Suppose a subset $S$ of Hadamard manifold $M$ of dimension $n$ is (geodesic) convex, and that $f:S \rightarrow \mathbb{R}$ is a (geodesic) convex $L$-Lipschitz  function.  Additionally assume the minimum of $f$, denote by $x_*$, is in the $\epsilon$-interior of $S$, meaning that the open ball centered at $x_*$ of Riemannian radius $\epsilon$ is contained within $S$.  Then it is possible to find a point $x \in S$ such that $f(x)-f(x_*) \leq \epsilon$ using $O(n^2 \log(n \, L \, \text{vol}_g(S) \epsilon^{-1}))$ subgradient oracle calls, where $\text{vol}_g(S)$ denotes the Riemannian volume of $S$.
\end{theorem}

\subsection{Definitions and Notation}\label{sec:defs}
In this section we give the definitions needed to frame the problem and results. Only basic notions of Riemannian geometry are needed in this paper; these are surveyed in Appendix \ref{app:riemannian}, with the present section mainly providing non-standard or less common definitions.  

\vspace*{.3cm}
For the remainder of this paper we study triples $(M, g, \mu)$, where $M$ is an n-dimensional, simply-connected manifold equipped with a complete Riemannian metric $g$ of non-positive sectional curvature. Such Riemannian manfolds $(M,g)$ are called Hadamard manifolds. For each point $x \in M$ we denote by $\langle \cdot, \cdot \rangle_x$ the inner product, defined by $g$, on the tangent space $T_xM$. The metric $g$ also induces the Riemannian volume measure, denoted by $\text{vol}_g$.  In addition to this, we consider a probability measure $\mu$, which we assume to be absolutely continuous with respect to $\text{vol}_g$. For the motivating application $\mu$ is taken to be the Riemannian volume measure restricted to a subset $S \subset M$, i.e. $\mu = \frac{\mathbbm{1}_{S}}{\text{vol}_g(S)} \cdot  \text{vol}_g $. The metric $g$ further induces a metric between points on the manifold, allowing us to define the notion of geodesics, which are locally length minimizing paths. For $x \in M$ we denote by $\exp_x: T_xM \rightarrow M$ the exponential map based at $x$, which maps tangent vectors to geodesics passing through $x$. As explained in the Riemannian geometry overview in Appendix \ref{app:riemannian}, the exponential map is a diffeomorphism when $(M,g)$ is a Hadamard manifold.

\vspace*{.3cm}
We proceed with definitions related to convexity,

\begin{definition}\label{def:convex_set}
We say that $S \subset M$ is \textbf{convex} if points $x, y \in S$ are joined by a unique length-minimizing geodesic contained in $S$.
\end{definition}

\begin{definition}\label{def:gconvex}
A function $f:M\rightarrow \mathbb{R}$ is \textbf{convex} on its domain if its restrictions to geodesics are convex in $t$.  That is, $f(\exp_x(t v)):\mathbb{R} \rightarrow \mathbb{R}$ is a convex function in $t$.

\vspace*{.3cm}
For such a convex function $f$, a tangent vector $w \in T_xM$ is said to be a \textbf{subgradient} at $x$ if for any $v\in T_xM$,
\[
f(\exp_x(t v)) \geq f(x) +  t \langle w , v \rangle_{x} \, .
\]
The set of subgradients at $x$ is known as the \textbf{subdifferential} at $x$, and is denoted by $\partial f_x$. 
\end{definition}

We note Theorem 4.5 in \cite{Udriste} proves that convex functions have a non-empty subdifferential at all points.  Accordingly, when the convex function $f$ is discussed, we will assume a subgradient oracle that for any $x$ outputs some $w \in \partial f_x$.  As explained in Appendix \ref{app:riemannian}, the gradient of a differentiable convex function is a subgradient.  Therefore the subradient oracle for such a function can be explicit. 

\vspace*{.3cm}
Let us present two examples of convex functions for motivation.  The first is general, the second specific.

\begin{itemize}
\item The distance to a convex subset (Definition \ref{def:convex_set}) of a Hadamard manifold is convex \cite{B}.  Thus finding the point minimizing the mean distance or mean squared-distance to a set of points is a convex optimization problem.

\item Identify $SL_n / SO_n$ with the set of positive-definite matrices of determinant 1 \cite{B}.  Then for arbitrary $B_i \in GL_n$, the function 
\[
\log \det \left (\sum\limits_{i=1}^m B_i^T X B_i \right)
\] defined on $SL_n / SO_n$ is convex \cite{SH}.  Minimizing such a function can be used to find the optimal Brascamp-Lieb constant \cite{BCCT}.  Up to scaling, all symmetric spaces of non-compact type (examples of Hadamard manifolds) embed as totally geodesic submanifolds of these spaces.  Therefore restrictions of this function to such submanifolds give many more examples.  Minimizing this function has figured prominently in recent theoretical computer science research, notably in \cite{AGLOW}.  Their work succeeds in developing an optimizaton procedure depending polynomially on $\log(\epsilon^{-1})$ for such functions, but the approach relies on special properties of this family of functions.
\end{itemize}

\vspace*{.3cm}
With these examples in mind, let us return for two more important definitions, 

\begin{definition}\label{def:halfspace}
An open \textbf{halfspace} based at $x \in M$ is formed by applying $\exp_x(\cdot)$ to a halfspace of $T_xM$. We denote halfspaces by
\[
H_x(v) := \{ \exp_x(w) \, | \, w \in T_xM, \, \langle w, v \rangle_x	< 0 \} \, ,
\]
for a given $v \in T_xM$.
\end{definition}

Although such halfspaces are not convex sets in the general setting of Hadamard manifolds, they are naturally produced by cutting planes for convex functions.  This notion of cutting plane is justified by the following lemma,

\begin{lemma}\label{lem:cutting_plane}
Consider a convex function $f:S \rightarrow \mathbb{R}$, where $S$ is a convex subset of Hadamard manifold $M$.  Then for any $x \in S$ and any subgradient $v \in \partial f_x$, the minimum of $f$ within $S$ is either attained at $x$ or lies within $H_x(v) \cap S$.  Moreover, if $y \in S \backslash H_x(v)$, then $f(y) \geq f(x)$.
\end{lemma}

\begin{proof}
If $y \notin H_x(v)$, the corresponding $v' = \exp_x^{-1}(y)$ satisfies 
\[
\langle v', v \rangle_x \geq 0 \, ,
\]
and we have 
\[
f(y) \geq f(c) + \langle v, v'\rangle_x \geq f(x) .
\]
\end{proof}

Cutting plane methods need to find a point for which no halfspace based at that point has too much of the feasible set's volume.  This can be captured through the notion of a centerpoint.  Our definition of centerpoint technically could be applied to any probability measure on any space with a notion of halfspace.  However, in proving Theorem \ref{thm:sharpness}, we further restrict to the halfspaces we defined for Hadamard manifolds, and require the probability measure to be absolutely continuous with respect to the Riemannian volume measure $\text{vol}_g$.

\begin{definition}\label{def:centerpoint}
A $\beta$-\textbf{centerpoint} of the probability measure $\mu$ on a Hadamard manifold $M$ is a point $c$ such that 
\[
\mu(H_c(v)) \leq 1-\beta \, ,
\]
for all $v \in T_c M$.
\end{definition}

Theorem \ref{thm:sharpness} claimed that even for the uniform measure on convex subsets of $M$, a $\frac{1}{n+1}$-centerpoint is the best we can guarantee. We included this comment both to contrast with $\mathbb{R}^n$, as well as to include a concrete illustration.  It is not difficult to present such an example through studying $\mathbb{H}^n$, the model space of constant $-1$ sectional curvature.

\vspace*{.3cm}
Towards this end, let us briefly state the important features of the Klein model of $\mathbb{H}^n$ that we require.  We identify $\mathbb{H}^n$ with the open Euclidean unit ball $B(1)\subset \mathbb{R}^n$, which we take to be centered at the origin.  This set is equipped with the metric $g = \frac{dx_1^2+\dots+dx_n^2}{1-x_1^2-\cdots-x_n^2}$.  This leads to a volume form of 
\[
\text{vol}_g = \frac{1}{(1-x_1^2-\cdots-x_n^2)^{\frac{n+1}{2}}}dx^1\wedge\cdots \wedge dx^n \, .
\]  

Critically for our exposition, this model of hyperbolic space has its geodesics appear as Euclidean lines; thus Riemannian halfspaces appear as halfspaces intersected with $B(1)$.  Other work such as \cite{BE} has found this useful in studying convex objects in $\mathbb{H}^n$.

\vspace*{.3cm}
For the construction demonstrating that $\frac{1}{n+1}$ cannot in general be improved in Theorem \ref{thm:sharpness}, the idea is simply that convex polyhedra in $\mathbb{H}^n$ have their volume concentrated towards the vertices.  Let $T(1)$ be the closed, regular $n$-simplex inscribed in $B(1)$.  It can be checked that $T(1)$ has finite volume; such objects are called ideal polyhedra and have been studied extensively. By symmetry, we will see that the origin $\vec{0}$ is the optimal centerpoint for $T(1)$.  However, note that a hyperplane through $\vec{0}$ parallel to any of the faces will contain exactly $1$ of the $n+1$ vertices.  An application of the Gauss-Bonnet theorem can be used to check that in the case of $\mathbb{H}^2$, the area of the halfspace containing a single vertex is $\frac{\pi}{3}$.  The halfspace containing the other two vertices is of area $\frac{2 \pi}{3}$. The author found this calculation to be easier to carry out in a conformal model such as the upper half-plane model, and conjectures the result to hold as well for higher dimensions.

\vspace*{.3cm}
For our purpose, it is more direct to modify the example slightly,

\begin{proposition}\label{prop:tight}
Let $T(1+\delta)$ be a closed, regular $n$-simplex inscribed in $B(1+\delta)$.  We take $\delta > 0$ small enough so that $B(1)$ is not inscribed in $T(1+\delta)$.  Also define $S_{\epsilon} = T(1+\delta) \cap B(1-\epsilon)$, which we observe to be convex and compact.  Then the optimal centerpoint of $S_{\epsilon}$ approaches being a $\frac{1}{n+1}$-centerpoint as $\epsilon \rightarrow 0$.

\end{proposition}
\begin{proof}
It is clear that $S=T(1+\delta)\cap B(1)$ has unbound Riemannian volume, because each of the $n+1$ vertices of $T(1+\delta)$ lies outside of $B(1)$.  However, $S_{\epsilon} = T(1+\delta) \cap B(1-\epsilon)$ has finite volume for any $\epsilon>0$, and is convex and compact.  Define the set of probability measures consisting of the uniform measure restricted to $S_{\epsilon}$, i.e. $\mu_{\epsilon} = \frac{\mathbbm{1}_{S_{\epsilon}}}{\text{vol}_g(S_{\epsilon})} \cdot  \text{vol}_g $.  

\vspace*{.3cm}
By symmetry, the origin $\vec{0}$ is the optimal centerpoint for each $\mu_{\epsilon}$.  In more detail, the optimal centerpoint for $S_{\epsilon}$ is the solution to minimizing $G(y):= \sup\limits_{\hat{v} \in S^{n-1}} \mu_{\epsilon}(H_y(\hat{v})$.  Viewing $G(y)$ as a function on $S_{\epsilon} \subset \mathbb{R}^n$, Lemma \ref{lem:quasiconvex} establishes that it is quasi-convex. Again, this is possible because in this model of $\mathbb{H}^n$, geodesics appear as Euclidean straight lines.  One characterization of quasi-convex is $f(tx + (1-t) y) \leq \max(f(x), f(y))$ for all $0 \leq t \leq 1$. Suppose $x_* \neq \vec{0}$ is the optimal centerpoint.  Because $S_{\epsilon}$ exhibits tetrahedral symmetry, we see there are $n+1$ optimal centerpoints, and their convex hull includes $\vec{0}$.  Using quasi-convexity, any points in the convex hull of these $n+1$ optimal centerpoints must also be optimal.  This proves $\vec{0}$ is the optimal centerpoint, using the symmetry of the set $S_{\epsilon}$ and quasi-convexity of the centerpoint function $G(y)$.

\vspace*{.3cm}
Now consider a hyperplane through $\vec{0}$ parallel to one of the faces of $S_{\epsilon}$, and denote by $H^+$ the resulting halfspace containing only one of the vertices of $T(1+\delta)$.  Because the volumes close to the vertices of $T(1+\delta)$ diverge at equal rates, it follows that as $\epsilon \rightarrow 0$, 
\[
\mu_{\epsilon}(H^+) = \text{vol}_g(H^+ \cap S_{\epsilon})/\text{vol}_g(S_{\epsilon})\rightarrow \frac{1}{n+1} .
\]
\end{proof}
This concludes our proof of the sharpness of Theorem \ref{thm:sharpness}.

\subsection{Overview and Conclusion}
The remainder of this paper is organized as follows:

\begin{itemize}
\item Section \ref{sec:centerpoint_existence} analyzes the existence of centerpoints on Hadamard manifolds.
\item Section \ref{sec:application} presents the brief application of the above to upper bound subgradient oracle complexity.
\item Appendix \ref{app:riemannian} recalls the relevant notions of Riemannian geometry and provides references.
\end{itemize}

To be clear, the problem of developing an efficient optimization procedure is far from resolved.  However, our results show that there is not an information theoretic obstacle to developing cutting plane methods for Hadamard manifolds.

\vspace*{.3cm}
We hope our main result is of interest and encourages others to study centerpoints in the manifold setting.  Targeting optimization procedures, we believe focusing on the spaces $SL_n / SO_n$ would be of greatest interest, both for theory and applications.  Computing a centerpoint from a discrete point set would be a notable advancement.  It would also be useful to be able to sample from the Riemannian volume restricted to a convex subset.

\section{Existence of Centerpoints}\label{sec:centerpoint_existence}
One might wonder if the centroid of a convex set of a manifold is an adequate centerpoint.  Here centroid refers to the center of mass, the point minimizing the average squared-distance.  After all, the centroid of a convex subset of $\mathbb{R}^n$ is an approximately optimal centerpoint \cite{G}.  However, this is tied closely to the fact that cross-sectional areas of a convex set in $\mathbb{R}^n$ follow a log-concave probability distribution - a consequence of the Brunn-Minkowski inequality.  On the otherhand, for a manifold with negative sectional curvature, the distribion of cross-sectional areas is not necessarily even unimodal. This reflects the fact that manifold versions of the Brunn-Minkowski inequality use curvature lower bounds as parameters, and are qualitatively different in negative curvature compared to $\mathbb{R}^n$ \cite{CE2001}. Helly's Theorem is somewhat the opposite, as it holds in situations in which the distance function is convex. Moreover, as cited in Appendix \ref{app:riemannian}, Hadamard manifolds have convex distance functions. One can find in \cite{LTZ} and \cite{I} proofs that amount to:

\begin{theorem}\label{Helly}
Let $M$ be an $n$-dimensional Riemannian manifold of non-positive sectional curvature. Suppose we are given a convex compact set $C$ and a family $\{C_{\alpha}\} \subset C$ of closed convex sets.  Then if for an arbitrary selection of $n+1$ sets $C_{\alpha_1} \cap \dots \cap C_{\alpha_{n+1}} \neq \emptyset$, it follows that $\cap_{\alpha} C_{\alpha} \neq \emptyset$
\end{theorem}

The paper \cite{I} actually proves this result for $\text{Cat}(0)$ geodesic spaces. 

\vspace*{.3cm}
That the halfspace notion of Definition \ref{def:halfspace} is not typically convex limits the applicability of this generalization of Helly's Theorem. The remainder of this section proves a result that could be a considered a Riemannian variant of the well known corollary of Helly's theorem cited as Theorem \ref{thm:classic}.  To generalize that result, we rely on a few simple regularity properties of sets of Euclidean centerpoints, which we now collect. In the following lemma, the halfspaces are Euclidean halfspaces, and $D$ is the Hausdorff distance. That is,
 \[
 D(A,B) := \max \{ \sup\limits_{b \in B}\, \inf \limits_{a \in A} |a-b|\, , \, \sup\limits_{a \in A}\, \inf \limits_{b \in B} |a-b| \},
 \]
 where $|\cdot|$ denotes the Euclidean norm.  In other words, $D(A,B)$ is the minimal value $\epsilon$ so that $A$ is contained in the $\epsilon$-fattened version of $B$ and vice versa.  Also recall the total variation distance between probability distributions is
 \[
 \sup\limits_{A \in \mathcal{F}} |\mu_1(A)-\mu_2(A)|,
 \]
where $|\cdot|$ denotes the absolute value, as there will be no confusion.  Here $A$ can be any measurable set, the collection of which is labeled $\mathcal{F}$.

\begin{lemma}\label{lem:quasiconvex}
Let $\{\mu_x(\cdot)\}$ be a family of probability measures on $\mathbb{R}^n$ that share a compact support $Y$. Assume the measures are indexed by members $x$ of a compact metric space $X$ with metric $d$, and the measures $\mu_x(\cdot)$ vary continuously with respect to total variation distance. Define the Euclidean centrality function $G: X \times Y \rightarrow \mathbb{R}$ by
\[
G(x, y) := \sup\limits_{\hat{v} \in S^{n-1}} \mu_x(H_y(\hat{v})),
\]
in order to measure how good of a centerpoint $y$ is for distribution $\mu_x$. Then $G$ is continuous under the product topology and $G(x, \cdot)$ is a quasi-convex function for a fixed $x$. Fixing an arbitrary $\alpha >0$, also define the sets
\[
U_x := \left\{y \in \mathbb{R}^n \, | \, G(x, y) \in \left(0, 1-\frac{1}{n+1}+\alpha \right] \right\}
\]
in order to explicitly propose the set of centerpoints for distribution $\mu_x$.  For any $x$, these sets have a non-empty interior.  Moreover fix $x \in X$ and suppose $\supp(\mu_x)$ is a connected set.  Then $x_i \rightarrow x$, $D(U_{x_i}, U_{x}) \rightarrow 0$.
\end{lemma}

\begin{proof}
Each $\{y \, | \, \mu_x(H_y(\hat{v})) < a \} $ is a halfspace. Indeed, there is a unique halfspace with normal $\hat{v}$ of mass $a$, and the previous set is precisely the points contained in this halfspace.  Therefore the intersection over all $\hat{v}$ is a convex set.  This shows that preimages under $G(x, \cdot)$ of sets $(-\infty, a)$ are convex, which is the definition of quasi-convex.

\vspace*{.3cm}
As we are using the product topology, the domain of $G$, which we denote by $K = X \times Y$, is compact. Because $g(x, y, \hat{v}): (x, y, \hat{v}) \mapsto \mu(H_x(\hat{v}))$ is continuous and $K$ is compact, $g$ is uniformly continuous on $K \times S^{n-1}$.  Thus given $\epsilon > 0$, one can choose $\delta$ so that when $d(x, x') < \delta$ and $|y-y'| < \delta$, then 
\[
|g(x, y, \hat{v}_0) - g(x', y', \hat{v}_0)| < \epsilon
\]  
holds for any $\hat{v}_0$. By compactness in the last argument, we may let $G(x, y) = g(x, y, \hat{v}_{x, y})$.  Therefore
\[
G(x, y)-G(x', y') = g(x, y, \hat{v}_{x, y}) - g(x', y', \hat{v}_{x', y'}) > g(x, y, \hat{v}_{x, y}) - (g(x, y, \hat{v}_{x', y'})+\epsilon) \geq - \epsilon 
\]
holds. Switching roles gives the reverse inequality, $G(x, y) - G(x', y') < \epsilon$, which proves continuity.

\vspace*{.3cm}
Recall the Hausdorff distance is the maximum distance it might require to travel from a point in one of the sets to the other set. We argue by contradiction that $U_{(\cdot)}$ converges to $U_x$ in the Hausdorff distance metric. Assume the contrary, then either (i) there exists a sequence of points $y_{n_i} \in U_{x_{n_i}}$ such that $y_{n_i}$ are bounded away from $U_x$ or (ii) there exists a sequence of points $y_{n_i} \in U_x$ bounded away from $U_{x_{n_i}}$.

In situation (i), compactness implies an accumulation point $y$ for the sequence $y_{n_i}$. However, continuity of $G$ requires $y \in U_x$, because $x_{n_i}\rightarrow x$ and each $G(x_{n_i}, y_{n_i}) \in (0, 1- \frac{1}{n+1}+\alpha]$.  This contradicts the premise that $y_{n_i}$ are bound away from $U_x$.  In particular, this shows that the maximum distance from a point in $U_{x_i}$ to the set $U_x$ is going to $0$. 

\vspace*{.3cm}
In situation (ii), again by compactness there is an accumulation point $y \in U_x$ that is bounded away from infinitely many of the $U_{x_{n_i}}$.  Because we have assumed $\alpha > 0$, Theorem \ref{thm:classic} and the continuity of $G(x, \cdot)$ imply $U_x$ has an interior.  Note that in proving continuity of $G(x, \cdot)$, we used the absolute continuity of $\mu_x$ with respect to Lebesgue measure.

\vspace*{.3cm}
As a first subcase of (ii), we assume $y$ is in the interior of $U_{x}$.  We show by contradiction that $G(x, y) < 1-\frac{1}{n+1}+\alpha$.  Supposing to the contrary, there would be $\hat{v}$ such that $G(x,y)= 1 - \frac{1}{n+1}+\alpha = \mu_x(H_y(\hat{v}))$, and we may select $p \in H_y(-\hat{v}) \cap U_x$ because $y$ is in the interior of $U_x$.  As $p \in U_x$ and $\mu_x(H_y(\hat{v})) = 1-\frac{1}{n+1}$, it must be the case that $\mu_x\left(H_y(-\hat{v}) \cap H_p(\hat{v})\right) = 0$, hence
\[
H_y(-\hat{v}) \cap H_p(\hat{v})\cap \supp(\mu_x)=\emptyset,
\]
and this implies the boundary of $H_{\frac{y+p}{2}}(\hat{v})$ is disjoint from $\supp(\mu_x)$.  Then $\supp(\mu_x)\cap H_{\frac{y+p}{2}}(-\hat{v})$ and $\supp(\mu_x)\cap H_{\frac{y+p}{2}}(\hat{v})$ are non-empty sets whose union is $\supp(\mu_x)$.  As these sets are open in the induced topology on $\supp(\mu_x)$, this contradicts the assumption that $\supp(\mu_x)$ is connected. Therefore we have shown by contradiction that $G(x, y) < 1-\frac{1}{n+1}+\alpha$.  We immediately conclude from the continuity of $G$ that $y \in U_{x_i}$ for large enough $i$.  

\vspace*{.3cm}
In the event that $y$ is not in the interior of $U_{x}$, we can still select an interior point $y' \in U_{x}$ that is arbitrarily close to $y$, because the set is open and convex.  For any such $y'$, the prior argument establishes that $y' \in U_{x_i}$ for large enough $i$.  Therefore, we conclude that the distance between $y$ and $U_{x_i}$ is going to $0$, contradicting our assumption.  This completes the proof showing $D(U_{x_{i}}, U_{x}) \rightarrow 0$.

\end{proof}

As a comment on the proof, the assumption that $\text{sup}(\mu_x)$ is connected was essential in the final conclusion of the proof. This assumption, along with a few others like the use of $\alpha$, are bootstrapped out of the eventual theorem we prove.  It would be nice to eliminate the absolute continuity assumption on $\mu_x$ by using a more general convergence tool like Wasserstein distance.  We necessarily lose $G$'s continuity, but it remains lower semi-continuous.  However, these topological properties alone were insufficient for proving $U_x$ had an interior point or analogous ``deep'' point, which we found necessary in proving Hausdorff convergence for $U_x$.

\vspace*{.3cm}
We are now ready for the key step in proving the main result.

\begin{proposition}\label{prop:cuts}
Let $\mu$ be a probability measure whose support lies within a compact set $S$ of a Hadamard manifold $M$.  Further assume $\mu$ is absolutely continuous with respect to the Riemannian volume measure $\text{vol}_g$ and has connected support.  Then there exists a $\frac{1}{n+1}$-centerpoint for $\mu$.
\end{proposition}

Before going into the proof details, here is conceptual overview of the proof.  We will define a continuous function $F$ from $S$ to itself, and an application of Brouwer's theorem will show there is a fixed point.  We design $F$ so that the fixed point is a $(\frac{1}{n+1}-\alpha)$-centerpoint.  The $\alpha$ is inherited from the definition of $U_x$ in Lemma \ref{lem:quasiconvex}, and is removed at the end of the proof.  In designing $F$, we adopt normal coordinates at $x$ and pull back the measure $\mu$ from $M$ (i.e. the measure of $U \subset \mathbb{R}^n$ is $\mu(\exp_x(U))$.  In these coordinates, there is a Euclidean-convex set of Euclidean centerpoints $U_x$ provided by the previous lemma, for the pulled-back measure.  We select the closest of these centerpoints to $x$ and denote this point by $u_x$. Finally, $F(x)$ is then defined by projecting $u_x$ onto $S$.  As stated precisely in the appendix, it is the Hadamard assumption that implies a strictly convex distance function, making this projection possible.

\vspace*{.3cm}
The technical part of the proof mostly involves showing continuity of $F(x)$, as it is not hard to show that fixed points are $(\frac{1}{n+1}-\alpha)$-centerpoints.  The main obstacle is to show that $u_x$ varies continuously.  To establish this, we note that the pulled back measures vary continuously with respect to total variation.  Then Lemma \ref{lem:quasiconvex} shows that the Euclidean centerpoint sets $U_x$, $U_{x'}$ are close in Hausdorff distance, provided $x, x'$ are close.  Combining this with convexity of the centerpoint sets, we are able to make $|u_x-u_x'|$ small.

\vspace*{.3cm}
We now provide the details.

\begin{proof}
We may WLOG assume $S$ is a closed Riemannian ball of radius $R$.  By parallel transport we may fix a smooth orthonormal frame $V= (\vec{e}_1, \dots \vec{e}_n)$ on $S$, thereby determining normal coordinate charts at each $x \in S$ defined by 
\[
\psi_x: y \mapsto \exp_x(y^i \vec{e}_i(x)).
\]  
Note that $\psi_x(y)$ varies smoothly both in $x$ and $y$. We may pull back the measure $\mu$ by $\psi_x$ to give the measures $\mu_x(y) dy$. The absolute continuity of these pull back measures with respect to Lebesgue measure is due to $\mu$ being absolutely continuous with respect to the Riemannian volume measure.  Smoothness of parallel transport ensures that the coordinate charts vary smoothly and therefore the $\mu_x(y)$ vary continuously with respect to total variation distance.  Finally, since $\psi_x$ are diffeomorphisms, we see that for each $x \in S$ the measure $\mu_x(y)$ has connected support. The set $S$ is of radius $R$.  Therefore in applying Lemma \ref{lem:quasiconvex}, we may choose $Y$ to be the closed ball of radius $2R$.

\vspace*{.3cm}
First fix $\alpha > 0$. For all $x \in S$, Lemma \ref{lem:quasiconvex} then establishes the existence of non-empty compact convex sets $U_x \subset \mathbb{R}^n$ of Euclidean $(\frac{1}{n+1}-\alpha)$-centerpoints.  There is a unique point $u_x \in U_x$ that is closest to $x$.  However, it is not necessarily the case that $u_x$ is inside $\psi_x^{-1}(S)$, because $\psi_x^{-1}(S)$ is not convex with respect to the Euclidean metric.  To work around this, project $u_x$ onto $S$.  That is,
\[
F(x) := \pi(u_x) := \argmin\limits_{s \in S} d(s, \psi_x(u_x)) \, ,
\]
where by $d(\cdot, \cdot)$ we mean the Riemannian distance.  The projection is well-defined and continuous by \cite[Corollary 5.6]{B}.  Therefore $F$ is well-defined.  In the following we show

\begin{itemize}
\item If $F(x) = x$, then $x$ is a $(\frac{1}{n+1}-\alpha)$-centerpoint contained in $S$.
\item $F(x)$ is continuous.
\end{itemize}
Then since $S$ is a closed ball, an application of Brouwer's fixed point theorem yields the desired result.  

\vspace*{.3cm}
We first show that fixed points are centerpoints.  We argue by contradiction and assume $x$ is a fixed point of $F$ which is not a $(\frac{1}{n+1}-\alpha)$-centerpoint. Observe that $u_x \neq \vec{0}$, because this would imply $x$ is a $(\frac{1}{n+1}-\alpha)$-centerpoint.  Further, we see that $\psi_x(H_0(-u_x))\cap S \neq \emptyset$, since $H_{u_x}(-u_x) \subset H_0(-u_x)$ and $\mu(\psi_x(H_{u_x}(-u_x)) > \frac{1}{n+1}-\alpha>0$ by the centerpoint property.  Next choose $s \in \psi_x(H_{\vec{0}}(-u_x)) \cap S$, and consider the geodesic between $x, s$.  This geodesic is contained in $S$ by the assumption that $S$ is convex.  Triangle inequalities in the form of Toponogov's Theorem \cite{CE2008} (or convexity of the distance function as a simple alternative) show that, initially, moving from $x$ to $s$ along the geodesic decreases the distance to $u_x$.  This means it is not the case that $\pi(\psi_x(u_x)) = x$.

\vspace*{.3cm}
Next we consider the continuity claim. Once we show $u_x \in \mathbb{R}^n$ varies continuously with respect to $x \in S$, then the continuity of $F(x)$ follows because, as noted in Appendix \ref{app:riemannian}, the projection is also continuous.  As a first step, we remark that the pull-back probability densities $\mu_x(y)$ vary continuously with respect $x$, because they are defined by smoothing varying diffeomorphisms $\psi_x$.  Moreover, as $S$ is compact and $\mu$ is supported on $S$, we may assume the $y$ are taken from a compact set.  Then we may apply uniform continuity to show there is $\delta$ so that $d(x,x') < \delta$ implies $|\mu_x(y) - \mu_{x'}(y)| < \epsilon$ for any $y$.  This establishes continuity for the family of measures $\mu_x(y)dy$, with respect to total variation distance.  We can now make use of the regularity properties provided by Lemma \ref{lem:quasiconvex}.

\vspace*{.3cm}
From the lemma's last part, by requiring $d(x, x') < \delta$ for small enough $\delta$, one can ensure $D(U_x, U_{x'}) < \epsilon$.  Let $h_{x} \in U_{x}$ be the point closest to $u_{x'}$; this ensures $|u_{x'} - h_{x}| < \epsilon$.  It is also not difficult to see that $|u_{x'}|-|u_x| < \epsilon$.  Therefore $|h_{x}|-|u_{x}| < 2\epsilon$.  Critically, the Euclidean distance to the origin is strongly convex and $u_x$ minimizes it on the Euclidean convex set $U_x$, which also includes $h_x$. Therefore, qualitatively, since $|h_x|$ and $|u_x|$ are similar in value, we know that $|h_x - u_x|$ is small.  Making this quantitative through the Euclidean law of cosines,
\[
|u_x - h_x|^2 \leq |h_x|^2-|u_x|^2 = (|h_x| - |u_x|)(|h_x| + |u_x|) < \epsilon R
\]
where a sufficiently large $R$ can be taken to be twice the diameter of $S$.  We conclude
\[
|u_x - u_{x'}| \leq | u_x - h_x| + |h_x - u_{x'}| < \sqrt{\epsilon R} + \epsilon \, ,
\]
which establishes continuity for $F(x)$.  This essentially completes the proof, but recall that we have used a small $\alpha$ parameter to define $U_x$, and this resulted in our proving only the existence of $(\frac{1}{n+1}-\alpha)$-centerpoints for $\alpha > 0$.  However, the continuity of the centerpoint function $\sup\limits_{\hat{v}\in S^{n-1}} H_x(\hat{v})$ on $S$ follows from the same argument for proving continuity of $G$ in Lemma \ref{lem:quasiconvex}. From this and compactness of $S$, may conclude the existence of $\frac{1}{n+1}$-centerpoints.
\end{proof}

This nearly proves the main part of Theorem \ref{thm:sharpness}.  The main difference is the absence of a few simplifying assumptions, namely compactness and connected support.  The fact that $x \in S$ provided $S$ is convex was also postponed.  We complete the proof here.

\begin{proof} [Proof for Theorem \ref{thm:sharpness}]
We first remove the connected support assumption. For measures $\mu$ supported on compact $S$, we may WLOG assume $S$ is a ball and therefore connected.  Then we may define the probability measures $(1-\epsilon)\mu + \epsilon \frac{\mathbbm{1}_{S}}{\text{vol}_g(S)} \cdot  \text{vol}_g $.  Proposition \ref{prop:cuts} applies to these measures. Thus it is clear that we may construct $(\frac{1}{n+1}-\epsilon)$-centerpoints for $\mu$.  Again using the continuity of the centerpoint function and compactness of $S$ as at the end of Proposition \ref{prop:cuts}'s proof, it follows that a $\frac{1}{n+1}$-centerpoint exists for $\mu$.

\vspace*{.3cm}
Next we extend the result by removing the assumption that $S$ is compact.  Fixing some point $x \in M$, we may define the family of compact sets $S_i = S \cap B_g(x,i)$, where $i$ ranges over the positive integers and $\bar{B}_g(x,i)$ denotes the closed ball of radius $i$ around $x$.  These sets satisfy $\lim\limits_{i\rightarrow \infty} \mu(S_i) = 1$.  Applying Proposition \ref{prop:cuts} to these $S_i$, we get points $s_i$ that are at least $\left(\frac{1}{n+1}-\mu(S_i^C)\right)$-centerpoints for $\mu$.  Moreover, these $s_i$ must all lie in some compact set $C \subset S$.  Indeed, by the analog of the separating hyperplane theorem proven in Lemma \ref{lem:sht}, any point $p \notin B_g(x,i)$ will have a halfspace $H_p(\hat{v}) \cap S_i = \emptyset$, and thus be at most a $\mu(S_i^C)$-centerpoint. As in Lemma \ref{lem:quasiconvex}, the function $G:C \rightarrow \mathbb{R}$ defined by $G(y):= \sup\limits_{\hat{v} \in S^{n-1}} \mu(H_y(\hat{v}))$ is continuous.  Because $\lim\limits_{i\rightarrow \infty} G(s_i) \leq 1-\frac{1}{n+1}$, $C$ is compact, and $G$ is continuous, it follows that there is some $c \in C \subset S$ such that $G(c) \leq 1-\frac{1}{n+1}$.  Therefore this $c$ is a $\frac{1}{n+1}$-centerpoint.

\vspace*{.3cm}
If we additionally assume $S$ is convex, then the separating hyperplane theorem again gives that the centerpoint satisfies $c \in S$.

\vspace*{.3cm}
Finally, we remind the reader that Proposition \ref{prop:tight} established the second part of the theorem, concerning sharpness.
\end{proof}
\section{Upper Bound on Needed Subgradient Calls}\label{sec:application}

We require one final lemma for the application to convex optimization.  In this lemma, as in the past, $\text{vol}_g$ will denote the Riemmanian volume measure and $B_g(x,r)$ the open ball of radius $r$ around $x$.

\begin{lemma}\label{lem:cut_ball}
Suppose $f$ is convex and $L$-Lipschitz on its convex domain $S \subset M$, where $M$ is a Hadamard manifold. Additionally assume the minimum $x_*$ is in the $\epsilon$-interior of $S$, meaning $B_g(x_*,\epsilon)\subset S$.  Now suppose we are given a sequence of cutting planes $H_{c_i}(v_i)$, $i = 1, 2, \cdots, N$ with $c_i \in S$ and $v_i \in \partial f_{c_i}$ such that the remaining feasible set $S' := S \cap_i H_{c_i}(v_i)$ satisfies the volume bound
\[
\text{vol}_g(S' := S \cap_i H_{c_i}(v_i)) < \frac{(\epsilon/L)^n}{n^n}.
\]
Then one of the $c_i$ satisfies $f(c_i) - f(x_*) \leq \epsilon $.
\end{lemma}

\begin{proof}
The main fact to be established is that $\text{vol}_g(B_g(x_*,\frac{\epsilon}{L})) > \text{vol}_g(S')$, as this implies that one of the complements of the halfspaces $H_{c_i}(v_i)$ must have intersected $B_g(x_*,\frac{\epsilon}{L})$.  By volume comparison methods (see \cite{CE2008}), the volume of this geodesic ball is greater or equal to the volume of a Euclidean ball of equal radius.  A reference justifying the exact version required is included in Appendix \ref{app:riemannian} as Theorem \ref{thm:bg}.  Hence we obtain 
\[
\text{vol}_g(B_g(x_*,\frac{\epsilon}{L})) > \frac{(\epsilon/L)^n}{n^n} > \text{vol}_g(S') ,
\]
by using $\frac{\pi^{\frac{n}{2}}}{\Gamma(\frac{n}{2}+1)} \frac{\epsilon^n}{L^n} > \frac{1}{n^n}\frac{\epsilon^n}{L^n}$ in the first inequality.
 
\vspace*{.3cm}
It follows that there exists a point $x' \in B(x_*, \frac{\epsilon}{L})$ that lies in the complement of one of the halfspaces $H_{c_i}(v_i)$.  From Lemma \ref{lem:cutting_plane}, $f(c_i) \leq f(x')$.  The Lipschitz bound on $f$ then gives
\[
f(c_i) - f(x_*) \leq f(x')-f(x_*) \leq L \cdot d(x', x_*) \leq \epsilon .
\]
\end{proof}

The proof of Theorem \ref{thm:oracle_cuts} is now a rather straightforward consequence.

\begin{proof}[Proof for Theorem \ref{thm:oracle_cuts}]
Lemma \ref{lem:cut_ball} shows that one of the origins of the cuts is $\epsilon$ from optimal for the function $f$ as soon as the remaining set, denoted by $S'$, has volume $O(\frac{\epsilon^n }{n^n L^n})$.

\vspace*{.3cm}
We must only bound the number of halfspaces needed to reduce the volume of $S'$ to this amount.  Proceeding iteratively, apply Theorem \ref{thm:sharpness} with $\mu$ being the Riemannian volume measure restricted to $S'$ (i.e. $\mu = \frac{\mathbbm{1}_{S'}}{\text{vol}_g(S')} \cdot  \text{vol}_g$).  As the support of $\mu$ is contained in the convex set $S$, Theorem \ref{thm:sharpness} shows that we may choose the cut centers to be $\frac{1}{n+1}$-centerpoints $c_i \in S$ for the remaining set $S' \subset S$, so that the volume is reduced by a factor $(1-\frac{1}{n+1})$ each cut.  This means the number of iterations needed is $O(n^2 \log\left(n L \, \text{vol}_g(S)\epsilon^{-1}) \right)$.
\end{proof}

\section{Acknowledgments}


Many people have listened patiently and offered advice which has been helpful for this work.  In particular, many thanks to Alexander Appleton, Richard Bamler, Andrew Hanlon, and Nikhil Srivastava.

\appendix
\section{Riemannian Overview}\label{app:riemannian}
We will be working in the setting of Riemannian geometry, but will not use much machinery.  We provide an informal overview.  The definitions we introduce here are generally standard and formalized in introductory texts, one such being \cite{L}.

\vspace*{.3cm}
An $n$-dimensional (smooth) manifold $M$ can be understood as a space that is locally diffeomorphic to $\mathbb{R}^n$, so we identify these subsets of $M$ with coordinates $(x_1, \dots, x_n)$.  This allows us to define smooth curves $\gamma:\mathbb{R}\rightarrow M$, by requiring their coordinate representations $(x_1(t), \dots, x_n(t))$ to be smooth.  We may define velocities $\gamma'(t)$ by associating them with $(x_1'(t), \dots, x_n'(t))$, leading to the notion of the tangent spaces $T_xM \cong \mathbb{R}^n$.

\vspace*{.3cm}
Riemannian manifolds additionally specify a metric for measuring the size of these velocities, by defining an inner product $\langle \cdot\, , \cdot \, \rangle_x$ on the tangent space of each $x \in M$.  This immediately enables the definition of curve length, as $\int |\gamma'(t)|_{\gamma(t)} dt$. It also gives a method of measuring volume; if $g_{ij}$ is the bilinear form for the metric in a local coordinate choice, then $\sqrt{|g|} dx^1\wedge\cdots\wedge dx^n$ is the Riemannian volume form.

\vspace*{.3cm}
It also turns out to be helpful to compute directional derivatives for vector fields (or acceleration along curves).  Requiring a few natural conditions leads to a unique Riemannian connection $\nabla: T_xM \times T_xM \rightarrow T_xM$ determined by the metric.  It is known as the Levi-Civita connection.  In the coordinates of a local frame $E = (\vec{e}_1, \dots, \vec{e}_n)$, which provides a basis for the tangent spaces of a neighborhood, the Riemannian connection is given by
\[
\nabla_{\vec{e}_i} \vec{e}_j = \Gamma_{ji}^k \vec{e}_k
\]
for the Christoffel symbols $\Gamma_{ij}^k$.  When the acceleration of a curve is $0$, i.e. $\nabla_{\gamma'(t)} \gamma'(t) \equiv 0$, we say that curve is a geodesic.  This is a second order non-linear ODE system for $\gamma(t) = (x_1(t), \dots, x_n(t))$,
\[
\ddot{x}^k(t) + \dot{x}^i\dot{x}^j \Gamma_{ij}^k(x(t)) = 0 .
\]
A unique solution will exist locally provided we specify the initial position and velocity.  The exponential map is defined by $\exp_p(v) = \gamma(1)$. For Hadamard manifolds, the exponential map is well-defined for any values of $p$ and $v$.  

\vspace*{.3cm}
At any any $x \in M$ we may consider the image of a tangent plane $\sigma_x$ spanned by $v, w \in T_x M$. Locally around $x$ the image is a surface.  The sectional curvature of the 2-plane $\sigma_x$ is defined to be the Gaussian curvature of the image surface at $x$.  Lower and upper bounds of the sectional curvature enable generalizations of Euclidean tools like ball volume and triangle trigonometry estimates.  The Bishop-Gromov volume comparison theorem is one important result along these lines.  Although usually stated for its volume upper bound by assuming just a lower bound on curvature, it is understood that the proof also provides a lower volume bound \cite{Petersen}.  Here we state a specialization of this theorem that suffices for our application,

\begin{theorem}[Bishop-Gromov]\label{thm:bg}
Suppose $M$ is a Hadamard manifold.  Let $\text{vol}_g$ denote the Riemannian volume and $B_g(x,r)$ denote the open ball of radius $r$ around $x$.  Then
\[
\text{vol}_g(B_g(x,r)) \leq \frac{\pi^{\frac{n}{2}}}{\Gamma(\frac{n}{2}+1)} r^n.
\]
The right side of this inequality is the volume of a Euclidean ball of radius $r$.
\end{theorem}

Hadamard manifolds are simply connected manifolds of non-positive sectional curvature.  They have been extensively studied in mathematical literature. We collect a few commonly used facts which we made use of or provide intuition.  For Hadamard manifolds,

\begin{itemize}
\item The exponential maps $\exp_x(\cdot)$ are diffeomorphisms from $T_xM$ to $M$ (Cartan-Hadamard theorem).

\item The distance to a point, $d(x, \cdot)$, is strictly convex.  The distance to a closed, convex set is convex.

\item Geodesics between points are unique and distance minimizing.

\item Projection onto closed, convex sets is well defined and continuous.
\end{itemize}

All of these properties can be found in \cite{B}.

\vspace*{.3cm}
In the proof of Theorem \ref{thm:sharpness}, we made use of the separating hyperplane theorem of convex geometry.  Here we briefly state and prove a version sufficient for this application.

\begin{lemma}\label{lem:sht}
Suppose $M$ is a Hadamard manifold, $S \subset M$ is a closed, convex set, and $p \notin S$.  Then there is a halfspace based at $p$ satisfying $H_p(v)\cap S = \emptyset$
\end{lemma}
\begin{proof}
Consider the function $f(x) = d(x,S)$. By \cite{B} this function is convex and hence for any $p \notin S$ there exists a subgradient $v \in \partial f_p$ (see Definition \ref{def:gconvex}). Then $H_p(v)$ is such a separating hyperplane, by Lemma \ref{lem:cutting_plane}.
\end{proof}

In the introduction, we mentioned that the gradient of a differentiable convex function is a subgradient.  We provide a short justification for this simple fact, as it is an important source of subgradients.  In Riemannian geometry, the gradient is defined by duality using the metric; that is, $\nabla$ satisfies $\langle \nabla f, \cdot \rangle = df(\cdot)$.
\begin{lemma}
Let $f(x):M\rightarrow \mathbb{R}$ be convex along geodesics as well as differentiable.  Then 
\[
f(y) \geq f(x) + \langle \nabla f(x), \exp_x^{-1}(y) \rangle_x
\]
\end{lemma}

\begin{proof}
Let $y = \exp_x(t_0 v)$.  That $f$ is convex on geodesics means $f(\exp_x(tv))$ is convex in $t$, so
\[
f(y) \geq f(x) + t_0 \frac{d}{dt}f(\exp_x(tv))|_0 \, .
\]
But using the chain rule and that $d(\exp_x)|_0 = I$ (see \cite{L}), 
\[
\frac{d}{dt}f(\exp_x(tv))|_0 = df(d \exp_x|_{\vec{0}}(v)) = df(v) = \langle \nabla f(x), v \rangle_x \, .
\]
\end{proof}



\bibliographystyle{plain}
\bibliography{references}

\end{document}